\documentclass{birkjour}
 \newtheorem{thm}{Theorem}[section]

 \theoremstyle{definition}
 
 \theoremstyle{remark}
 \newtheorem{rem}[thm]{Remark}
 
 \numberwithin{equation}{section}

\newcommand{\Mk}{\mathcal{M}_k}
\newcommand{\Mkk}{\mathcal{M}_{k-1}}
\newcommand{\Hk}{\mathcal{H}_k}
\newcommand{\Rm}{\mathbb{R}^m}
\newcommand{\C}{\mathbb{C}}
\newcommand{\Clm}{\mathcal{C}l_m}

\newcommand{\Sm}{\mathbb{S}^{m-1}}

\newcommand{\Dodd}{\mathcal{D}_{2j-1}}

\begin{document}

%
%
%
%
%
%
%
%
%

\title[Integral Formulas for Higher Order Fermionic Operators]
 {Integral Formulas for Higher Order Conformally Invariant Fermionic Operators}

\author{Chao Ding}

\address{%
Bauhaus-Universit\"{a}t Weimar\\
Institute of Mathematics/Physics\\
 Coudraystr. 13B\\
  99423 Weimar\\
 Germany}

\email{chaoding1985@gmail.com}

\thanks{This research is supported by Bauhaus-Postdoc-$\text{Profil}^+$-Scholarship 2018/19.}


\date{}
\dedicatory{Communicated by Vladim\'{i}r Sou\v{c}ek}

\begin{abstract}
In this paper, we establish higher order Borel-Pompeiu formulas for conformally invariant fermionic operators in higher spin theory, which is the theory of functions on $m$-dimensional Euclidean space taking values in arbitrary irreducible representations of the Spin group. As applications, we provide higher order Cauchy's integral formulas for those fermionic operators. This paper continues the work of building up basic integral formulas for conformally invariant differential operators in higher spin theory. 
\end{abstract}

\maketitle
\section{Introduction}
\indent\indent
Classical Clifford analysis starts from \cite{Brackx} the study of Dirac type operators and monogenic functions (null solutions of the Dirac operator), which are generalizations of holomorphic functions in one dimensional complex analysis. One of the most important results is the Borel-Pompeiu formula, which leads to various higher dimensional analogs of well-known results in one dimensional complex analysis. For instance, Cauchy's integral formula, Cauchy theorem, etc. Borel-Pompeiu formulas are also important for solving certain types of boundary value problems for partial differential equations, for instance, \cite{Begehr3,Shapiro}. In the framework of Clifford analysis, higher order integral formulas were studied by many authors, see \cite{Begehr1,Begehr2,Begehr3,Reyes,JohnCauchy,Zhang}.
\par
In the past few decades, many authors \cite{B1,Bures,Ding0,Ding3,D,Eelbode} have been working on generalizations of classical Clifford analysis to the so-called higher spin theory. This investigates higher spin operators acting on functions on $\Rm$ taking values in irreducible representations of the Spin group. In Clifford analysis, these irreducible representations are traditionally constructed as spaces of homogeneous polynomials satisfying certain differential equations. A first order conformally invariant differential operator, called Rarita-Schwinger operator, was first studied systematically in \cite{Bures} and revisited in \cite{D} with different techniques. In both papers, integral formulas, such as Borel-Pompeiu formula, Cauchy's integral formula, were established. A second order conformally invariant differential operator, called the generalized Maxwell operator or the higher spin Laplace operator, was constructed recently in \cite{B1,Eelbode}. Some basic integral formulas for the higher spin Laplace operator were also found recently, see \cite{Ding1}. In \cite{Ding3}, the authors studied the construction of arbitrary order conformally invariant differential operators in higher spin spaces. These are fermionic operators when the orders are odd and bosonic operators when the orders are even. However, integral formulas for higher order ($\geq 3$) conformally invariant differential operators have not been found yet. In this paper, we will continue this work by establishing Borel-Pompeiu formulas for higher order fermionic operators. As applications, higher order Cauchy's integral formulas for fermionic operators are provided.
\par
This paper is organized as follows. In Section 2, we introduce the framework of Clifford analysis and some standard facts in higher spin theory. For instance, definitions and fundamental solutions for fermionic operators are provided here. Section 3 is devoted to a summary of Stokes' Theorems for Rarita-Schwinger type operators, which are the main tools used in this paper. In Section 4, we demonstrate higher order Borel-Pompeiu formulas for fermionic operators and higher order Cauchy's integral formulas are provided as a special case.

\section{Preliminaries}
\indent\indent
Let $\{e_1,e_2,\cdots,e_m\}$ be an orthonormal basis for the $m$-dimensional Euclidean space $\Rm$. The real Clifford algebra is generated by these basis elements with the defining relations $$e_i e_j + e_j e_i= -2\delta_{ij},\ 1\leq i,j\leq m,
$$ where $\delta_{ij}$ is the Kronecker delta function. An arbitrary element of the basis of the Clifford algebra can be written as $e_A=e_{j_1}\cdots e_{j_r},$ where $A=\{j_1, \cdots, j_r\}\subset \{1, 2, \cdots, m\}$ and $1\leq j_1< j_2 < \cdots < j_r \leq m.$
Hence for any element $a\in \mathcal{C}l_m$, we have $a=\sum_Aa_Ae_A,$ where $a_A\in \mathbb{R}$. The complex Clifford algebra $\Clm (\C)$ is defined as the complexification of the real Clifford algebra
$$\Clm (\C)=\Clm\otimes_{\mathbb{R}}\C.$$
We consider real Clifford algebra $\Clm$ throughout this subsection, but in the rest of the paper we consider the complex Clifford algebra $\Clm (\mathbb{C})$ unless otherwise specified. We define $\widetilde{e_{j_1}\cdots e_{j_r}}=e_{j_r}\cdots e_{j_1}$ and $\overline{e_{j_1}\cdots e_{j_r}}=(-1)^re_{j_r}\cdots e_{j_1}$, which are two types of Clifford anti-involutions, see \cite{Brackx} for more details.
\par
The Pin and Spin groups play important roles in Clifford analysis. We define $$Pin(m)=\{a\in \mathcal{C}l_m: a=y_1y_2\dots y_{p},\ y_1,\dots,y_{p}\in\mathbb{S}^{m-1},p\in\mathbb{N}\},$$ 
where $\mathbb{S} ^{m-1}$ is the unit sphere in $\Rm$. $Pin(m)$ is clearly a multiplicative group in $\mathcal{C}l_m$ and it is actually a double cover of $O(m)$. Indeed, suppose $v\in \mathbb{S}^{m-1}\subseteq \mathbb{R}^m$, and if we consider $vxv$, we may decompose
$$x=x_{v\parallel}+x_{v\perp},$$
where $x_{v\parallel}$ is the projection of $x$ onto $v$ and $x_{v\perp}$ is the remainder part, which is perpendicular to $v$. Hence $x_{v\parallel}$ is a scalar multiple of $v$ and we have
$$vxv=vx_{v\parallel}v+vx_{v\perp}v=-x_{v\parallel}+x_{v\perp}.$$
So the action $vxv$ describes a reflection of $x$ in the direction of $v$. By the Cartan-Dieudonn$\acute{e}$ Theorem each $O\in O(m)$ is the composition of a finite number of reflections. If $v=y_1\cdots y_p\in Pin(m),$ then $\tilde{v}:=y_p\cdots y_1$ and we observe $vx\tilde{v}=O_v(x)$ for some $O_v\in O(m)$. Choosing $y_1,\ \dots,\ y_p$ arbitrarily in $\mathbb{S}^{m-1}$, we have that the group homomorphism
\begin{align*}
\theta:\ Pin(m)\longrightarrow O(m)\ :\ v\mapsto O_v,
\end{align*}
with $v=y_1\cdots y_p$ and $O_vx=vx\tilde{v}$ is surjective. Further, since $-vx(-\tilde{v})=vx\tilde{v}$, so $1,\ -1\in \ker(\theta)$. In fact $\ker(\theta)=\{1,\ -1\}$. See \cite{P1}. The Spin group is defined as
$$Spin(m)=\{a\in \mathcal{C}l_m: a=y_1y_2\dots y_{2p},\ y_1,\dots,y_{2p}\in\mathbb{S}^{m-1},p\in\mathbb{N}\}$$
 and it is a subgroup of $Pin(m)$. With the same group homomorphism $\theta$ defined above, one can see that $Spin(m)$ is the double cover of $SO(m)$. See \cite{P1} for more details.
\par
The Dirac operator in $\mathbb{R}^m$ is given by $D_x:=\sum_{i=1}^{m}e_i\partial_{x_i}$, where $\partial_{x_i}$ stands for the partial derivative with respect to $x_i$. Notice that $D_x^2=-\Delta_x$, where $\Delta_x$ is the Laplacian in $\mathbb{R}^m$.  A $\Clm$-valued function $f(x)$ defined on a domain $U$ in $\Rm$ is left monogenic if $D_xf(x)=0.$ Sometimes, we will consider the Dirac operator $D_u$ in a variable $u$ rather than $x$.
\par 
Let $\mathcal{M}_k$ denote the space of $\mathcal{C}l_m$-valued monogenic polynomials homogeneous of degree $k$, and $\Hk$ stands for the space of  $\mathcal{C}l_m$-valued harmonic polynomials homogeneous of degree $k$, then we have the following well-known \emph{Almansi-Fischer decomposition} of $\Hk$, see \cite{D} for more details.
$$\mathcal{H}_k=\mathcal{M}_k\oplus u\mathcal{M}_{k-1}.$$
In this Almansi-Fischer decomposition, we have $P_k^+$ and $P_k^-$ as the projection maps 
\begin{align*}
&P_k^+=1+\frac{uD_u}{m+2k-2}: \mathcal{H}_k\longrightarrow \mathcal{M}_k,\\
&P_k^-=I-P_k^+=\frac{-uD_u}{m+2k-2}:\ \mathcal{H}_k\longrightarrow u\mathcal{M}_{k-1}.
\end{align*}
Suppose $U$ is a domain in $\mathbb{R}^m$, we consider a differentiable function $f: U\times \mathbb{R}^m\longrightarrow \mathcal{C}l_m$
such that, for each $x\in U$, $f(x,u)$ is a left monogenic polynomial homogeneous of degree $k$ in $u$. Then the Rarita-Schwinger operator \cite{Bures,D} is given by 
 $$R_k=P_k^+D_x:\ C^{\infty}(\Rm,\Mk)\longrightarrow C^{\infty}(\Rm,\Mk).$$
 We also have the following three more Rarita-Schwinger type operators.
 \begin{align*}
	&\text{\textbf{The twistor operator:}}\\
	& T_k=P_k^+D_x:\ C^{\infty}(\Rm,u\Mkk)\longrightarrow C^{\infty}(\Rm,\Mk),\\
	&\text{\textbf{The dual twistor operator:}}\\
	& T_k^*=P_k^-D_x:\ C^{\infty}(\Rm,\Mk)\longrightarrow C^{\infty}(\Rm,u\Mkk),\\
	&\text{\textbf{The remaining operator:}}\\
	& Q_k=P_k^-D_x:\ C^{\infty}(\Rm,u\Mkk)\longrightarrow C^{\infty}(\Rm,u\Mkk).
	\end{align*}
	See \cite{Bures,D} for more details. These operators can also be constructed as Stein-Weiss gradients and they are special cases of the higher spin Dirac operator, see \cite{Des,EeRa}.
	\par
	The $(2j-1)$-th order ($j>1$) conformally invariant differential operator in higher spin theory, called the $(2j-1)$-th order fermionic operator \cite{Ding3}, is given by
	\begin{align*}
\Dodd=R_k\prod_{s=1}^{j-1}\big(a_sT_kT_k^*+b_sR_k^2\big),
\end{align*}
where 
\begin{align*}
&a_s=\frac{-4s^2}{(m+2k-2s-2)(m+2k+2s-2)},\ 1\leq s\leq j-1,\\
&b_s=1,\ 1\leq s\leq j-1.
\end{align*}
The fundamental solution for $\Dodd$ is provided in the same reference. That is,
\begin{align*}
E_k^{2j-1}(x,u,v)=\lambda_{2j-1}\frac{x}{||x||^{m-2j+2}}Z_k\bigg(\frac{xux}{||x||^2},v\bigg),
\end{align*}
where $\lambda_{2j-1}$ is a nonzero constant given in \cite{Ding3}, and $Z_k(u,v)$ is the reproducing kernel for $\Mk$, which satisfies
\begin{align*}
f(v)=\int\displaylimits_{\Sm}\overline{Z_k(u,v)}f(u)dS(u),\ \text{for\ any}\ f(v)\in\Mk. 
\end{align*}
In particular, $E_k^1(x,u,v)$ is the fundamental solution for the Rarita-Schwinger operator $R_k$.
To conclude this section, we point out that, if $D_uf(u)=0$ then $\overline{D_uf(u)}=-\overline{f(u)}D_u=0$. Hence, we can define right monogenic functions, right monogenic polynomials with homogeneity of degree $k$, right Almansi-Fischer decomposition of $\Hk$, etc. In other words, for all the results we introduced above, we have their analogs for right monogenic functions.
\section{Stokes' Theorem for Rarita-Schwinger Type Operators}
For convenience, we review Stokes' Theorem for Rarita-Schwinger type operators as follows. More details can be found in \cite{Bures,Ding0,D}.
\begin{thm}\cite{D}\textbf{(Stokes' Theorem for $R_k$)}\label{StokesRk}\\
Let $\Omega '$ and $\Omega$ be domains in $\mathbb{R}^m$ and suppose the closure of $\Omega$ lies in $\Omega '$. Further suppose the closure of $\Omega$ is compact and $\partial \Omega$ is piecewise smooth. Let $f,g\in C^{\infty}(\Omega ',\Mk)$. Then
\begin{align*}
&\int\displaylimits_{\Omega}\big[(g(x,u)R_k,f(x,u))_u+(g(x,u),R_kf(x,u))_u\big]dx^m\\
=&\int\displaylimits_{\partial\Omega}(g(x,u),d\sigma_xf(x,u))_u.
\end{align*}
where $d\sigma_x=n(x)d\sigma(x)$, $d\sigma(x)$ is the area element. 
\begin{align*}
(P(u),Q(u))_u=\int\displaylimits_{\mathbb{S}^{m-1}}P(u)Q(u)dS(u)
\end{align*}
 is the inner product for any pair of $\Clm$-valued polynomials.
\end{thm}
It is worth pointing out that we shortened the version in \cite{D} by omitting $P_k^+$ in the last two equations on purpose. This can also be found in the proof of this theorem in \cite{D}. Note also that in the Stokes' Theorem above, for $g(x,u)R_k$, the $R_k$ is the right Rarita-Schwinger operator. Since it is positioned on the right hand side, there should be no confusion for this. The reader will see more such terms in the rest of this paper.
\begin{thm}\cite{Ding0}\textbf{(Stokes' Theorem for $T_k$ and $T_k^*$)}\label{StokesTk}\\
Let $\Omega$ and $\Omega '$ be defined as above. Then for $f\in C^{\infty}(\Rm,\Mk)$ and $g\in C^{\infty}(\Rm,u\Mkk)$, we have
\begin{align*}
&\int\displaylimits_{\partial\Omega}\big(g(x,u),d\sigma_xf(x,u)\big)_u\\
=&\int\displaylimits_{\Omega}\big(g(x,u)T_k,f(x,u)\big)_udx^m+\int\displaylimits_{\Omega}\big(g(x,u),T_k^*f(x,u)\big)_udx^m.
\end{align*}
\end{thm}
\begin{thm}\cite{Li}\textbf{(Stokes' Theorem for $Q_k$)}\\
Let $\Omega$ and $\Omega '$ be defined as above. Then for $f,g\in C^{\infty}(\Rm,u\Mkk)$, we have
\begin{align*}
&\int\displaylimits_{\Omega}\big[(g(x,u)Q_k,f(x,u))_u+(g(x,u),Q_kf(x,u))_u\big]dx^m\\
=&\int\displaylimits_{\partial\Omega}(g(x,u),d\sigma_xf(x,u))_u.
\end{align*}
\end{thm}
We also omitted $P_k^-$ in the last equation for convenience. See \cite{Li} for more details.
\section{Integral Formulas for Fermionic Operators}
Recall that $\Dodd$ given in Section $2$ is
\begin{align*}
\Dodd=R_k\prod_{s=1}^{j-1}\big(a_sT_kT_k^*+b_sR_k^2\big),\ j\geq 2.
\end{align*}
For convenience, we let $\mathcal{D}_1=R_k$ and notice that $\Dodd$ is in terms of the Rarita-Schwinger type operators. This suggests us to apply the Stokes' Theorem for the Rarita-Schwinger type operators below.
	\begin{thm}\textbf{(Higher Order Borel-Pompeiu Formula)}\\
	Let $\Omega$ and $\Omega'$ be domains in $\Rm$ and suppose that the closure of $\Omega$ lies in $\Omega'$. Further, suppose the closure of $\Omega$ is compact  and $\partial\Omega$ is piecewise smooth. Suppose that $f(x,u)\in C^{\infty}(\Rm,\Mk)$ and $y\in\Omega^0$. Then
	\begin{align*}
	&\int\displaylimits_{\Omega}\big(E_k^{2j-1}(x-y,u,v),\Dodd f(x,u)\big)_udx^m\\
	=&\sum_{t=2}^{j}\int\displaylimits_{\partial\Omega}\big(E_k^{2t-1}(x-y,u,v),d\sigma_x(a_tT_k^*+b_tR_k)\mathcal{D}_{2t-3}f(x,u)\big)_u\\
	&-\sum_{t=2}^j\int\displaylimits_{\partial\Omega}\big(E_k^{2t-1}(x-y,u,v)T_k^*,d\sigma_xa_t\mathcal{D}_{2t-3}f(x,u)\big)_u\\
	&-\sum_{t=2}^j\int\displaylimits_{\partial\Omega}\big(E_k^{2t-1}(x-y,u,v)R_k,d\sigma_xb_t\mathcal{D}_{2t-3}f(x,u)\big)_u\\
	&+\int\displaylimits_{\partial\Omega}\big(E_k^1(x-y,u,v),d\sigma_xf(x,u)\big)_u-f(y,v).
	\end{align*}
	\end{thm}
	\begin{proof}
	Since $R_k$ and $b_sR_k^2+a_sT_kT_k^*$ commute (see \cite{Ding3}), we have $\Dodd=(a_{j-1}T_kT_k^*+b_{j-1}R_k^2)\mathcal{D}_{2j-3}$. This indicates that we can prove the theorem above by induction.\\
	Let $B_r=\{x\in\Omega:\ ||x-y||<r\}\subset \Omega$ for some sufficiently small $r>0$ and $B_r^c=\Omega\backslash B_r$. For convenience, we denote $E_k^{2j-1}=E_k^{2j-1}(x-y,u,v)$ unless it is necessary to specify its dependence on the variables. If $j=2$, we have
	\begin{align}
	&\int\displaylimits_{\Omega}\big(E_k^{3},\mathcal{D}_3 f(x,u)\big)_udx^m\nonumber\\
	=&\int\displaylimits_{B_r}\big(E_k^{3},\mathcal{D}_3 f(x,u)\big)_udx^m+\int\displaylimits_{B_r^c}\big(E_k^{3},\mathcal{D}_3 f(x,u)\big)_udx^m.\label{odd1}
	\end{align}
	Now, we consider the first integral
	\begin{align*}
	&\int\displaylimits_{B_r}\big(E_k^{3}(x-y,u,v),\mathcal{D}_3 f(x,u)\big)_udx^m\\
	=&\int\displaylimits_{B_r}\int\displaylimits_{\Sm}\frac{x-y}{||x-y||^{m-2}}Z_k(\frac{(x-y)u(x-y)}{||x-y||^2},v)\mathcal{D}_3 f(x,u)dS(u)dx^m.
	\end{align*}
	Since the homogeneity of $x-y$ in the integrand is $3-m$, we can see that the integral above goes to zero when $r$ goes to zero. Therefore, we only need to deal with the second integral in (\ref{odd1}). That is,
	\begin{align*}
	\int\displaylimits_{B_r^c}\big(E_k^{3},\mathcal{D}_3 f(x,u)\big)_udx^m=\int\displaylimits_{B_r^c}\big(E_k^{3},(a_1T_kT_k^*+b_1R_k^2)R_k f(x,u)\big)_udx^m.
	\end{align*}
	Now, we apply the Stokes' Theorems for $T_k$ and $R_k$ to the integral above, it becomes
	\begin{align*}
	&\int\displaylimits_{\partial\Omega}\big(E_k^{3},d\sigma_x(a_1T_k^*+b_1R_k)R_k f(x,u)\big)_u\\
	&-\int\displaylimits_{\partial B_r}\big(E_k^{3},d\sigma_x(a_1T_k^*+b_1R_k)R_k f(x,u)\big)_u\\
	&-\int\displaylimits_{B_r^c}\big(E_k^{3}T_k^*,a_1T_k^*R_k f(x,u)\big)_udx^m
	-\int\displaylimits_{B_r^c}\big(E_k^{3}R_k,b_1R_kR_k f(x,u)\big)_udx^m\\
	=&\int\displaylimits_{\partial\Omega}\big(E_k^{3},d\sigma_x(a_1T_k^*+b_1R_k)R_k f(x,u)\big)_u-\int\displaylimits_{B_r^c}\big(E_k^{3}T_k^*,a_1T_k^*R_k f(x,u)\big)_udx^m\\
	&-\int\displaylimits_{B_r^c}\big(E_k^{3}R_k,b_1R_kR_k f(x,u)\big)_udx^m.
	\end{align*}
	The integral on $\partial B_r$ vanishes because of the homogeneity of $x-y$ in $E_k^3$, which indicates that the integral goes to zero when $r$ goes to zero. Now, we apply the Stokes' Theorems to the last two integrals on $B_r^c$. We will omit the integrals over $\partial B_r$ below, since they also go to zero for the same reason we already mentioned. 
	\begin{align*}
	&\int\displaylimits_{\partial\Omega}\big(E_k^{3},d\sigma_x(a_1T_k^*+b_1R_k)R_k f(x,u)\big)_u-\int\displaylimits_{\partial\Omega}\big(E_k^{3}T_k^*,d\sigma_xa_1R_k f(x,u)\big)_u\\
	&+\int\displaylimits_{B_r^c}\big(E_k^{3}T_k^*T_ka_1,R_k f(x,u)\big)_udx^m-\int\displaylimits_{\partial\Omega}\big(E_k^{3}R_k,d\sigma_xb_1R_k f(x,u)\big)_u\\
	&+\int\displaylimits_{B_r^c}\big(E_k^{3}R_k^2b_1,R_k f(x,u)\big)_udx^m\\
	=&\int\displaylimits_{\partial\Omega}\big(E_k^{3},d\sigma_x(a_1T_k^*+b_1R_k)R_k f(x,u)\big)_u-\int\displaylimits_{\partial\Omega}\big(E_k^{3}T_k^*,d\sigma_xa_1R_k f(x,u)\big)_u\\
	&+\int\displaylimits_{B_r^c}\big(E_k^{3}(T_k^*T_ka_1+R_k^2b_1),R_k f(x,u)\big)_udx^m\\
	&-\int\displaylimits_{\partial\Omega}\big(E_k^{3}R_k,d\sigma_xb_1R_k f(x,u)\big)_u.
	\end{align*}
	In \cite{Ding3}, we have shown that $(a_{j-1}T_kT_k^*+b_{j-1}R_k^2)E_k^{2j-1}(x,u,v)=E^{2j-3}(x,u,v)$ for integers $j>1$. Therefore, the integral above over $B_r^c$ becomes
	\begin{align*}
	&\int\displaylimits_{B_r^c}\big(E_k^1(x-y,u,v),R_k f(x,u)\big)_udx^m\\
	=&\int\displaylimits_{\partial\Omega}\big(E_k^1(x-y,u,v),d\sigma_x f(x,u)\big)_u-f(y,v).
	\end{align*}
	It is worth pointing out that this equation is actually Theorem $7$ in \cite{D}. Hence, we showed that
	\begin{align*}
	&\int\displaylimits_{\Omega}\big(E_k^{3},\mathcal{D}_3 f(x,u)\big)_udx^m\\
	=&\int\displaylimits_{\partial\Omega}\big(E_k^{3},d\sigma_x(a_1T_k^*+b_1R_k)R_k f(x,u)\big)_u-\int\displaylimits_{\partial\Omega}\big(E_k^{3}T_k^*,d\sigma_xa_1R_k f(x,u)\big)_u\\
	&-\int\displaylimits_{\partial\Omega}\big(E_k^{3}R_k,d\sigma_xb_1R_k f(x,u)\big)_u+\int\displaylimits_{\partial\Omega}\big(E_k^1(x-y,u,v),d\sigma_x f(x,u)\big)_u\\
	&-f(y,v),
	\end{align*}
	which is our theorem with $j=2$.\\
	Assume that our theorem is true for $\mathcal{D}_{2j-3}$. Now we consider the case $\Dodd$, the argument is very similar as in the case $j=2$. First, we have
	\begin{align*}
	&\int\displaylimits_{\Omega}\big(E_k^{2j-1},\Dodd f(x,u)\big)_udx^m\nonumber\\
	=&\int\displaylimits_{B_r}\big(E_k^{2j-1},\Dodd f(x,u)\big)_udx^m+\int\displaylimits_{B_r^c}\big(E_k^{2j-1},\Dodd f(x,u)\big)_udx^m.
	\end{align*}
	 Since the first integral goes to zero when $r$ goes to zero with a similar argument as for (\ref{odd1}), we only need to deal with the second integral, that is
	 \begin{align*}
	 &\int\displaylimits_{B_r^c}\big(E_k^{2j-1},\Dodd f(x,u)\big)_udx^m\\
	 =&\int\displaylimits_{B_r^c}\big(E_k^{2j-1},(a_jT_kT_k^*+B_jR_k^2)\mathcal{D}_{2j-3} f(x,u)\big)_udx^m.
	 \end{align*}
	 Then we apply the Stokes' Theorems for $T_k$ and $R_k$ to the integral above, we have
	 \begin{align*}
	 &\int\displaylimits_{\partial\Omega}\big(E_k^{2j-1},d\sigma_x(a_jT_k^*+b_jR_k)\mathcal{D}_{2j-3} f(x,u)\big)_u\\
	 &-\int\displaylimits_{\partial B_r}\big(E_k^{2j-1},d\sigma_x(a_jT_k^*+b_jR_k)\mathcal{D}_{2j-3} f(x,u)\big)_u\\
	&-\int\displaylimits_{B_r^c}\big(E_k^{2j-1}T_k^*,a_jT_k^*\mathcal{D}_{2j-3} f(x,u)\big)_udx^m\\
	&-\int\displaylimits_{B_r^c}\big(E_k^{2j-1}R_k,b_jR_k\mathcal{D}_{2j-3} f(x,u)\big)_udx^m\\
	=&\int\displaylimits_{\partial\Omega}\big(E_k^{2j-1},d\sigma_x(a_jT_k^*+b_jR_k)\mathcal{D}_{2j-3} f(x,u)\big)_u\\
	&-\int\displaylimits_{B_r^c}\big(E_k^{2j-1}T_k^*,a_jT_k^*\mathcal{D}_{2j-3} f(x,u)\big)_udx^m\\
	&-\int\displaylimits_{B_r^c}\big(E_k^{2j-1}R_k,b_jR_k\mathcal{D}_{2j-3} f(x,u)\big)_udx^m.
	 \end{align*}
	 The integral over $\partial B_r$ vanishes, since it goes to zero when $r$ goes to zero because of the homogeneity of $x-y$ in $E_k^{2j-1}$. We next apply Stokes' Theorems to the last two integrals over $B_r^c$ to obtain
	 \begin{align}
	&\int\displaylimits_{\partial\Omega}\big(E_k^{2j-1},d\sigma_x(a_jT_k^*+b_jR_k)\mathcal{D}_{2j-3} f(x,u)\big)_u\nonumber\\
	&-\int\displaylimits_{\partial\Omega}\big(E_k^{2j-1}T_k^*,d\sigma_xa_j\mathcal{D}_{2j-3} f(x,u)\big)_u\nonumber\\
	&+\int\displaylimits_{B_r^c}\big(E_k^{2j-1}T_k^*T_ka_j,\mathcal{D}_{2j-3} f(x,u)\big)_udx^m\nonumber\\
	&-\int\displaylimits_{\partial\Omega}\big(E_k^{2j-1}R_k,d\sigma_xb_j\mathcal{D}_{2j-3} f(x,u)\big)_u\nonumber\\
	&+\int\displaylimits_{B_r^c}\big(E_k^{2j-1}R_k^2b_j,\mathcal{D}_{2j-3} f(x,u)\big)_udx^m\nonumber\\
	=&\int\displaylimits_{\partial\Omega}\big(E_k^{2j-1},d\sigma_x(a_jT_k^*+b_jR_k)\mathcal{D}_{2j-3} f(x,u)\big)_u\nonumber\\
	&+\int\displaylimits_{B_r^c}\big(E_k^{2j-1}(T_k^*T_ka_j+R_k^2b_j),\mathcal{D}_{2j-3} f(x,u)\big)_udx^m\nonumber\\
	&-\int\displaylimits_{\partial\Omega}\big(E_k^{2j-1}T_k^*,d\sigma_xa_j\mathcal{D}_{2j-3} f(x,u)\big)_u\nonumber\\
	&-\int\displaylimits_{\partial\Omega}\big(E_k^{2j-1}R_k,d\sigma_xb_j\mathcal{D}_{2j-3} f(x,u)\big)_u.\label{odd2}
	\end{align}
	Recall that $(a_jT_kT_k^*+b_jR_k^2)E_k^{2j-1}=E_k^{2j-3}$ \cite{Ding3}, with the assumption of the truth for the case $\mathcal{D}_{2j-3}$, the last integral becomes
	\begin{align*}
	&\int\displaylimits_{B_r^c}\big(E_k^{2j-3},\mathcal{D}_{2j-3} f(x,u)\big)_udx^m\\
	=&\sum_{t=2}^{j-1}\int\displaylimits_{\partial\Omega}\big(E_k^{2t-1},d\sigma_x(a_tT_k^*+b_tR_k)\mathcal{D}_{2t-3}f(x,u)\big)_u\\
	&-\sum_{t=2}^{j-1}\int\displaylimits_{\partial\Omega}\big(E_k^{2t-1}T_k^*,d\sigma_xa_t\mathcal{D}_{2t-3}f(x,u)\big)_u\\
	&-\sum_{t=2}^{j-1}\int\displaylimits_{\partial\Omega}\big(E_k^{2t-1}R_k,d\sigma_xb_t\mathcal{D}_{2t-3}f(x,u)\big)_u\\
	&+\int\displaylimits_{\partial\Omega}\big(E_k^1,d\sigma_xf(x,u)\big)_u-f(y,v).
	\end{align*}
	Plugging this back into (\ref{odd2}), we obtain
	\begin{align*}
	&\int\displaylimits_{\Omega}\big(E_k^{2j-1},\Dodd f(x,u)\big)_udx^m\\
	=&\sum_{t=2}^{j}\int\displaylimits_{\partial\Omega}\big(E_k^{2t-1},d\sigma_x(a_tT_k^*+b_tR_k)\mathcal{D}_{2t-3}f(x,u)\big)_u\\
	&-\sum_{t=2}^{j}\int\displaylimits_{\partial\Omega}\big(E_k^{2t-1}T_k^*,d\sigma_xa_t\mathcal{D}_{2t-3}f(x,u)\big)_u\\
	&-\sum_{t=2}^{j}\int\displaylimits_{\partial\Omega}\big(E_k^{2t-1}R_k,d\sigma_xb_t\mathcal{D}_{2t-3}f(x,u)\big)_u\\
	&+\int\displaylimits_{\partial\Omega}\big(E_k^1,d\sigma_xf(x,u)\big)_u-f(y,v).
	\end{align*}
	This completes the proof for the case $\Dodd$ and the theorem.
 	\end{proof}
	In particular, if we assume $\Dodd f=0$, we can obtain higher order Cauchy's integral formulas immediately.
	\begin{thm}\textbf{(Higher Order Cauchy's Integral Formula)}\\
	Let $\Omega$ and $\Omega'$ be defined as in the previous theorem. Suppose that $f(x,u)\in C^{\infty}(\Rm,\Mk)$, $\Dodd f=0$ and $y\in\Omega^0$. Then
	\begin{align*}
	f(y,v)=&\sum_{t=2}^{j}\int\displaylimits_{\partial\Omega}\big(E_k^{2t-1}(x-y,u,v),d\sigma_x(a_tT_k^*+b_tR_k)\mathcal{D}_{2t-3}f(x,u)\big)_u\\
	&-\sum_{t=2}^j\int\displaylimits_{\partial\Omega}\big(E_k^{2t-1}(x-y,u,v)T_k^*,d\sigma_xa_t\mathcal{D}_{2t-3}f(x,u)\big)_u\\
	&-\sum_{t=2}^j\int\displaylimits_{\partial\Omega}\big(E_k^{2t-1}(x-y,u,v)R_k,d\sigma_xb_t\mathcal{D}_{2t-3}f(x,u)\big)_u\\
	&+\int\displaylimits_{\partial\Omega}\big(E_k^1(x-y,u,v),d\sigma_xf(x,u)\big)_u.
	\end{align*}
	\end{thm}
	\begin{rem}
	We notice that the argument here does not work for the higher order bosonic operator cases, since bosonic operators act on functions taking values in scalar-valued homogeneous harmonic polynomials, and the Almansi-Fischer decomposition fails in this circumstance. We shall discuss the bosonic cases in future work with different techniques.
	\end{rem}
\subsection*{Acknowledgment}
The author is grateful to the anonymous referees for detailed comments.

\end{document}